\documentclass[12pt]{amsart} %amsartにしたい
\usepackage{amsthm, amssymb, color}
\usepackage{amsmath}
\usepackage{txfonts}
\usepackage{epsfig,multicol}

\usepackage{pict2e}

\theoremstyle{plain} \numberwithin{equation}{section}
\newtheorem{theo}{Theorem}[section]

\newtheorem{lemm}[theo]{Lemma}
\theoremstyle{definition}
\newtheorem{defi}[theo]{Definition}
\newtheorem{rema}[theo]{Remark}
\newtheorem{exam}[theo]{Example}
\newtheorem{claim}[theo]{Claim}

\def\mG{\mathcal G}

\def\mHT{H_T}
\def\Z{\mathbb Z}

\def\R{\mathbb R}

\def\l{\ell}
\def\t{t}
\def\s{s}
\def\w{w}
\def\u{u}
\def\si{\sigma}

\def\mA{\mathcal A}

\def\mG{\mathcal G}

\DeclareMathOperator{\Map}{Map}
\DeclareMathOperator{\rank}{rank}

\thanks{The author was supported by JSPS Research Fellowships for Young Scientists.}

\begin{document}
\title[graph cohomology ring of a flag manifold of type $G_2$]{The graph cohomology ring of the GKM graph of a flag manifold of type $G_2$}
\author[Y. Fukukawa]{Yukiko Fukukawa}  %\author[H. Ishida]{Hiroaki Ishida} \author[M. Masuda]{Mikiya Masuda}
\address{Department of Mathematics, Osaka City University, Sumiyoshi-ku, Osaka 558-8585, Japan.}
\email{yukiko.fukukawa@gmail.com} %\email{hiroaki.ishida86@gmail.com} \email{masuda@sci.osaka-cu.ac.jp}

\date{\today}
\maketitle %タイトル表示

 \begin{abstract}
 Suppose a compact torus $T$ acts on a closed smooth manifold $M$. Under certain conditions, Guillemin and Zara associate to $(M, T )$  a labeled graph $\mG_M$ where the labels lie in $H^2(BT)$. 
They also define the subring $H_T^*(\mG_M)$ of $\bigoplus_{v\in V(\mG_M)}H^*(BT)$, where $V(\mG_M)$ is the set of vertices of $\mG_M$ and we call $H_T^*(\mG_M)$ the \lq\lq graph cohomology" ring of $\mG_M$. It is known that the equivariant cohomology ring of $M$ can be described by using combinatorial data of the labeled graph. The main result of this paper is to determine the ring structure of equivariant cohomology ring of a flag manifold of type $G_2$ directly, using combinatorial techniques on the graph $\mG_M$. This gives a new computation of the equivariant cohomology ring of a flag manifold of type $G_2$.   (See \cite{deander}.)
% If a closed smooth manifold $M$ with an action of a torus $T$ satisfies certain conditions, then a labeled graph $\mG_M$ with labeling in $H^2(BT)$ is associated with $M$, which encodes a lot of geometrical information on $M$.  For instance, the \lq\lq cohomology" ring $H_T^*(\mG_M)$ of $\mG_M$ is defined to be a subring of $\bigoplus_{v\in V(\mG_M)}H^*(BT)$, where $V(\mG_M)$ is the set of vertices of $\mG_M$,  and is known to be often isomorphic to the equivariant cohomology $H^*_T(M)$ of $M$. In this paper, we determine the ring structure of $\mHT^*(\mG_M)$ when $M$ is a flag manifold of $G_2$ directly without using the fact that $\mHT^*(\mG_M)$ is isomorphic to $H^*_T(M)$.
 \end{abstract}

\section{Introduction}\label{intro}

 Suppose that a closed smooth manifold $M$ has an action of a compact torus $T$. 
 If the $T$-action is \lq\lq nice", then it is known that we can describe its equivariant cohomology ring by using combinatorial way, namely by using GKM theory. By \lq\lq nice" we mean that the $T$-action on $M$ is GKM, namely the fixed point set of the $T$-action $M^T$ is a finite set and the equivariant one-skeleton of $M^T$ is a union of points or 2-spheres. Then we construct a graph $\mG_M$ by replacing fixed points with vertices and 2-spheres with edges. 
This graph equipped with more information is defined by Guillemin and Zara \cite{gu-za01} to be the GKM graph associated with $(M, T)$.  
The object we study in this manuscript, the flag manifold of type $G_2$ with a standard maximal torus action, is GKM. 

Interestingly, we can describe the equivariant cohomology ring of $M$ by using the data of the GKM graph.  
 The equivariant cohomology ring of $M$ is defined to be the ordinary cohomology ring of Borel construction of $M$, namely $$H^*_T(M\ ;\mathbb{Z}):=H^*(ET\times _T M \ ; \mathbb{Z})$$ 
 where $ET$ is the total space of the universal principal $T$-bundle $ET \rightarrow BT$.  
 Since $T$ acts on $ET$, we can consider the diagonal $T$-action on $ET \times M$, and its orbit space $ET \times _T M$ is called the Borel construction of $M$.  
 In this paper, we treat the equivariant cohomology ring with $\mathbb{Z}$-coefficients, so we abbreviate $H^*_T(M\ ; \mathbb{Z})$ as $H^*_T(M)$.  
 Since a GKM space satisfies the condition that $H^*_T(M)$ is torsion free as a module over $H^*(BT)$, the restriction map $$\iota^* : H^*_T(M) \rightarrow  H^*_T(M^T)$$ is injective. 
Moreover, $M^T$ is a finite set, hence $M^T$ is isolated, so we have 
$$H^*_T(M^T) \cong \bigoplus_{p \in M^T}H^*_T(p) \cong \bigoplus _{p \in M^T} H^*_T(BT) \cong  \bigoplus_{p \in M^T}\mathbb{Z}[\t_1, \cdots, \t_n],$$
where $n = \dim T$. 
Therefore, we shall regard $H^*_T(M)$ as a subring of $\bigoplus_{p \in M^T}\mathbb{Z}[\t_1, \cdots, \t_n]$ through the map $\iota^*$.  
Guillemin and Zara defined the subring, denoted by $H^*_T(\mathcal{G}_M)$, of $ \bigoplus_{p \in M^T}\mathbb{Z}[\t_1, \cdots, \t_n]$, by using the combinatorial data of the graph $\mG_M$.  
Then, according to the result of Goresky-Kottwitz-MacPherson in \cite{go-ko-ma98}, we have 
$$H^*_T(M) \otimes \mathbb{Q} \cong H^*_T(\mathcal{G}_M) \otimes \mathbb{Q}.$$ 
If $M$ is a flag manifold then $H^*_T(M)$ is isomorphic to $H^*_T(\mathcal{G}_M)$ without tensoring with $\mathbb{Q} $ (see \cite{ha-he-ho05}, for example). 
Namely, by using data of the GKM graph, we can compute $H^*_T(M)$. 

Our purpose is to determine the ring structure of $H^*_T(\mathcal{G}_M)$ for a flag manifold $M$ of different Lie types. 
 In the paper \cite{fu-i-ma}, we computed the ring structure of $H^*_T(\mathcal{G}_M)$ for a flag manifold $M$ of classical type  directly, namely without using  the fact that $H^*_T(\mathcal{G}_M) \cong H^*_T(M)$, and our computation of $H^*_T (\mG_M)$
confirms that $H^*_T (M)$ is isomorphic to $H^*_T (\mG_M)$. 
The goal of this paper is to similarly determine the ring structure of $H^*_T(\mathcal{G}_M)$ for a flag manifold of type $G_2$. 
The ring structure of $H^*_T(\mathcal{G}_M)$ is given by the following theorem. 
\begin{theo}\label{premaintheo}
Let $\mG_2$ be the labeled graph associated with the flag manifold of type $G_2$. Then 
\begin{equation*}
H^*_T(\mG_2) = \Z[\tau_1,\tau_2,\tau_3,\t_1,\t_2,\t_3,f]/I,
\end{equation*}
where $I=(e_1(\tau), e_2(\tau)-e_2(\s), 2f-e_3(\tau)-e_3(\s), f^2-fe_3(\s) )$, and $e_i(\tau)$ $($resp. $e_i(s)$ $)$ is the $i^{th}$ elementary symmetric polynomial in $\tau_1,\tau_2,\tau_3$ 
$($resp. $s_1:=\t_1-\t_2 ,s_2:=\t_2-\t_3 ,s_3:=\t_3-\t_1$ $)$.
\end{theo}

To prove Theorem~\ref{premaintheo} we will use the computation of the graph cohomology of type $A_2$, because the labeled graph of type $G_2$ contains two labeled subgraphs isomorphic to the labeled graph of type $A_2$.  
In fact, the labeled graph of type $G_2$ can be viewed as the total space of a \lq\lq GKM fiber bundle" in the sense of Guillemin, Sabatini and Zara \cite{gu-sa-za} where the type $A_2$ subgraph is the fiber, although we do not use this perspective in our computation. 

This paper is organized as follows.   
In Section 2, we recall the labeled graph of a flag manifold. 
The labeled graph is a graph with weight attached to each edge. 
%The labeled graph is a graph obtained \textcolor{green}{from a GKM graph} by dropping the information which is unnecessary for our purposes.  %\textcolor{red}{スペース}% from the GKM graph to describe the ring structure. 
In Section 3, we give a different description of the Weyl group of type $G_2$ (the vertex set of the labeled graph) which allows us to describe the labeled graph more concretely. 
In Section 4, we compute the graph cohomology ring of the labeled graph.  

% In section 2, we define the labeled graph. The labeled graph is a graph obtained by dropping the information which is unnecessary for our purposes. 
% In section 3, we consider the description of the Weyl group of type $G_2$ (the vertex set of the labeled graph) to draw the labeled graph concretely. 
% In section 4, we define the equivariant cohomology ring of the labeled graph and compute its ring structure. 

% We will use the computation in \cite{fu-i-ma} of the graph cohomology of type $A_2$, because the labeled graph of type $G_2$ contain two isomorphism copies of 
% the labeled graph of type $A_2$. 
% In fact, the labeled graph of type $G_2$ can be viewed as the total space of a \lq\lq GKM fiber bundle" in the sense of Guillemin, Sabatini and Zara \cite{gu-sa-za} where the type $A_2$ subgraph is the fiber, although we do not use this perspective in our computation. 

\section{The labeled graph $\mG_M$}\label{sec2}
 In this section, we recall the definition of the labeled graph $\mG_M$ %and its cohomology ring
 for a flag manifold $M$. % and review the result in  \cite{fu-i-ma} when $M$ is a flag manifold of type $A_2$%When $M$ is a flag manifolds, the labeled graph $\mG_M$ and its cohomology ring $\mHT^*(\mG_M)$ is given by the following:
For an $n$-dimensional torus $T$, let $\{ \t_i \}_{i=1}^n$ be a basis of $H^2(BT)$, so that $H^*(BT)$ can be identified with the polynomial ring $\Z[\t_1, \t_2, \cdots , \t_n]$.
 We take an inner product on $H^2(BT)$ such that the basis $\{ \t_i \}_{i=1}^n$ is orthonormal. 
 The following is a simplified version of the definition of GKM graph given in \cite{gu-za01}. To distinguish our graph from theirs, we call ours a labeled graph. 

\begin{defi}(See 2.2, \cite{gu-ho-za06})\label{def2}
Let $M$ be a flag manifold of classical type or exceptional type, namely a flag manifold $M$ is a homogeneous space $G/T$ where $G$ is a compact Lie group and $T$ is a maximal torus of $G$. 
Suppose that $\Phi(G)$ be the root system of that type  and $W(G)$ be the Weyl group. (We regard  $\Phi(G)$  as a subset of $H^2(BT)= \{\sum_{i=1}^{n} a_i \t_i  \mid \  a_i \in \Z \}$. ) The labeled graph $\mG_M$ has $W(G)$ as a vertex set. Two vertices $\w$ and $\w'$ in $W(G)$ are connected by an edge $e$ if and only if there is an element $\alpha$ in $\Phi(G)$ such that $\w=\w' \si_\alpha$, where $\si_\alpha$ is the reflection determined by $\alpha$. The label of the edge $e$, denoted by $\l(e)$, is given by $\w \alpha$. 
 \end{defi} 
\begin{rema}\label{signoflabel}
Guillemin and Zara \cite{gu-za01} introduced the notion of a GKM graph which is a graph equipped with an axial function which satisfies a certain compatibility condition.  They defined the graph cohomology for a GKM graph, but this definition does not use the compatibility condition of the axial function.  One can also see that the  graph cohomology is independent of the signs of the labels. Therefore, we omit the axial function in the data in definition~\ref{def2}, and will often disregard the signs of the labels on our labeled graph. 

% The GKM graph given contains the information about the orientation of edges, but we do not need that information when we consider the equivariant cohomology ring of graph. 
%  So we omit the information about orientation from the original definition and we call our graph  a labeled graph. 
%  Especially, the definition of the equivariant cohomology is independent of the sign of the label. Therefore, we will often disregard the data of the sign of the root (namely, the label).  
 \end{rema}

\begin{exam}[$A_2$ type]
Let $M$ be the flag manifold of type $A_{2}$, namely $M=U(3)/T$ where $T$ is the maximal torus of $U(3)$. Then, the root system $\Phi(A_{2})$ is $\{ \pm(\t_i-\t_j) \ \mid \ 1 \leq  i <  j  \leq  3 \}$ and the Weyl group is the permutation group $S_3$ on three letters. We use the one-line notation $v=v(1)v(2)v(3)$ for permutations.  We denote by $\mathcal{A}_{3}$ the labeled graph associated with $\Phi(A_{2})$. It is shown in Figure~\ref{a3}.
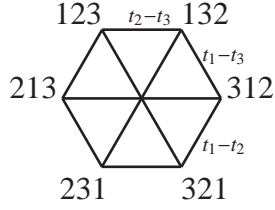
\begin{figure}
\begin{center}
\begin{picture}(130,50)
 \linethickness{1pt}
 \put(5,0){\line(100,173){15}}
 \put(20,25.95){\line(1,0){30}}
 \put(50,25.95){\line(100,-173){15}}
 \put(65,0){\line(-100,-173){15}}
 \put(50,-25.95){\line(-1,0){30}}
 \put(5,0){\line(100,-173){15}}
 \put(5,0){\line(1,0){60}}
 \put(20,25.95){\line(200,-346){30}}
 \put(20,-25.95){\line(200,346){30}}
  \put(-15,0){213}
  \put(2,28){123}
  \put(50,28){132}
  \put(67,0){312}
  \put(50,-38){321}
  \put(4,-38){231}
  \put(30,29){$\scriptstyle\t_2-\t_3$}
  \put(58,14){$\scriptstyle\t_1-\t_3$}
  \put(58,-20){$\scriptstyle\t_1-\t_2$}
%  \put(-5,-55){The labeled graph $\mA_3$}
\end{picture}
 \ \\ \  \\ \  
 \caption{$\mathcal{A}_3$}
 \label{a3}
\end{center}
 \end{figure}
\end{exam}

\section{Labeled graph of type $G_2$}\label{sectgraph}
In this section we concretely describe the labeled graph of the flag manifold associated to the compact Lie group of exceptional type $G_2$. Then the root system of type $G_2$ is known to be  
\begin{equation*} 
\Phi(G_2) := \{\pm (\t_{i_1} - \t_{i_2}), \pm (2\t_i -\t_j - \t_k) \  \mid 1 \leq i_1 < i_2 \leq 3, \{  i,j,k \}=[3] \}, 
\end{equation*}
where $[3]:=\{1,2,3 \}$.   
Let $\s_1=\t_1-\t_2$, $\s_2=\t_2 -\t_3$ and $\s_3 =\t_3 -\t_1$, then %the root system of type $G_2$ is written by 
\begin{equation*}
 \Phi(G_2) = \{\pm \s_k, \pm (\s_i -\s_j) \ | \ k \in [3], 1 \leq i < j \leq 3 \}, 
 \end{equation*}
so it is easy to see that $\Phi(G_2)$ has $\Phi(A_2)$ as a subset (but $s_i$'s play a role of $t_i$'s).  
We denote by $\mG_2$ the labeled graph associated with $\Phi(G_2) $. The graph $\mG_2$ has the Weyl group $W(G_2)$ of type $G_2$ as the vertex set. 
Let $\alpha_1=\s_1$ and $\alpha_2=\s_3-\s_1$ be the simple roots, then $W(G_2)$ has a presentation 
\begin{equation}\label{eq1}
<\sigma_1, \sigma_2 \mid {\sigma_1}^2={\sigma_2}^2=(\sigma_1 \sigma_2)^6=1>,
\end{equation} 
where $\sigma_i$ is the reflection defined by $\alpha_i$ for $i=1,2$. It is a dihedral group of order 12. 

We shall give another description of $W(G_2)$ as a set not as a group, which turns out to be convenient for our purpose, and rewrite the condition about edges and labels. 
Let $$\Phi=\{ \pm (\s_i -\s_j) \mid 1 \leq i < j \leq 3  \} \ \subset \Phi(G_2).$$ 
Let $W(\Phi)$ be the reflection group determined by $\Phi$, namely %and $\rho :=(\si_1 \si_2)^3$. Then 
%$W(\Phi)$ is written by 
$$W(\Phi)=<\si_1 \si_2 \si_1, \si_2> $$
because the reflections $\sigma_{ s_i-s_j}$ determined by the roots $ s_i-s_j$ in $\Phi$ are given by $$\sigma_{ s_3-s_1}=\sigma_2, \  
\sigma_{ s_1-s_2} =\sigma_1 \sigma_2 \sigma_1 \quad \text{and} \quad  \sigma_{ s_2-s_3}=(\sigma_1 \sigma_2)^3 \sigma_1$$　
It follows from \eqref{eq1} that the relations
$${\si_2}^2= (\si_1 \si_2 \si_1)^2=(\si_1 \si_2 \si_1 \cdot   \si_2)^3=1$$
hold, so we can identify $W(\Phi)$ with $W(A_3)=S_3$.  We choose a group isomorphism $\psi$ between  $W(\Phi)$ and $S_3$ as follows;
\begin{equation}\label{psi} 
\begin{array}{ccccc}
\psi  & : & W(\Phi) &\cong & S_3 \\
\  & \  & \si_{ s_i-s_j} & \longmapsto & (i,j)
\end{array}
\end{equation} 
where $(i,j)$ is the transposition of $i$ and $j$.
We note that $$W(G_2)=W(\Phi) \amalg \rho W(\Phi) \quad \text{as a set,}$$ where $\rho :=(\si_1 \si_2)^3$. (Note that $\rho$ is the rotation by angle $\pi$. ) 
We record the preceding discussion in a lemma.  
\begin{lemm}
Let $\Psi$ be the map from $W(G_2)$ to $S_3\times \{\pm\}$ defined as follows; for any $\w$ in $W(\Phi)$, 
$$\Psi (\w):=(\psi(w), +) \quad \text{and} \quad \Psi (\rho \w):=(\psi(\w), -).$$ 
Then $\Psi$ is bijective, so that one can identify $W(G_2)$ with $S_3\times\{\pm\}$ as a set  through the map $\Psi$. 
\end{lemm}
By using the bijection $\Psi$, we can concretely describe the edge and the label of the graph $G_2$.  
% Remember the relation between edges and labels. 
% \begin{quote}
% There is an edge $e_{\w, \w'}$ which connects two vertices $\w$ and $\w'$ in $W(G_2)$ if and only if there is an element $\alpha$ of $\Phi(G_2)$ such that $\w=\w'\si_{\alpha}$, and the label of $e_{\w, \w'}$ is $\w \alpha$.
% \end{quote}
The following lemma tells us the way to find the label $\w \alpha$ in the Definition~\ref{def2} more concretely. 
\begin{lemm} \label{lem1}
For any $\w_1$ and $\w_2$ in $W(G_2)$ connected by an edge $e_{\w_1, \w_2}$ labeled by $\w_1 \alpha$ for some $\alpha $ in $\Phi(G_2)$, namely $\w_1=\w_2 \si_\alpha$, one of the following occurs.\\
Case 1: both $\w_1$ and $\w_2$ are in $W(\Phi)$. In this case there are distinct integers $i$ and $j$ in $[3]$
 such that $\psi(\w_1)(i)=\psi(\w_2)(j), \psi(\w_1)(j)=\psi(\w_2)(i)$ and $\l(e_{\w_1,\w_2})=\s_{\psi(\w_1)(i)}-\s_{\psi(\w_1)(j)}$. \\
Case 2: both $\w_1$ and $\w_2$ are in $\rho W(\Phi)$ so that there are unique elements $\w_k'$ in $W(\Phi)$ such that $\w_k= \rho \w_k'$ for $k=1,2$. In this case there are distinct integers $i$ and $j$ in $[3]$
 such that $\psi(\w_1')(i)=\psi(\w_2')(j), \psi(\w_1')(j)=\psi(\w_2')(i)$ and $\l(e_{\w_1,\w_2})=\s_{\psi(\w_1')(i)}-\s_{\psi(\w_1')(j)}$.\\
Case 3: one of $\w_1$ and $w_2$ is in $W(\Phi)$ and the other is in $\rho W(\Phi)$. Without loss of generality, we may assume $\w_1 \in W(\Phi)$. Then there is an element $\w_2'$ in $W(\Phi)$ such that $\w_2=\rho \w_2'$. In this case there are distinct integers $i$ and $j$ in $[3]$
 such that $\psi(\w_1)(i)=\psi(\w_2')(j), \psi(\w_1)(j)=\psi(\w_2')(i)$ and $\l(e_{\w_1,\w_2})=\s_{\psi(\w_1)(k)}$, where $k \in [3] \setminus \{ i,j \}$.
\end{lemm} 
\begin{proof}
From the definition of $\psi$, we have $\psi(\si_{s_i-s_j})=(i,j)$ for $\{i,j\} \subset [3]$. 
Since the graph cohomology is independent of the signs of the label, we do not need to be careful of signs of labels of the labeled graph when we consider the graph cohomology ring of the labeled graph. Therefore, in this proof, we sometime disregard signs in front of roots.  

First, we prove the lemma for cases 1 and 2. In these cases $\alpha$ is in $\Phi$,  so $\alpha=s_i-s_j$ for some distinct $i, \ j$ in $[3]$.

Case 1. By assumption $\w_1$ and $\w_2$ are in $W(\Phi)$. Remember that $\w_1=\w_2 \si_{\alpha}$ and $\alpha =s_i-s_j$. Since $\psi $ is a group isomorphism, we have 
\begin{equation*}
\psi(\w_1) = \psi(\w_2)\psi(\si_{\alpha}) =\psi(\w_2)\psi(\sigma_{s_i-s_j})= \psi(\w_2)(i,j).
\end{equation*} 
Therefore $\psi(\w_1)(i)=\psi(\w_2)(j), \psi(\w_1)(j)=\psi(\w_2)(i)$. 
In addition,  since $\alpha =s_i- s_j$, we have  
$\ell (e_{w_1, w_2})=\w_1\alpha=\w_1(s_i-s_j)$, so in order to show $\ell (e_{w_1 ,w_2})=s_{\psi (w_1)(i)}-s_{\psi (w_1)(j)}$, it is enough to show that 
\begin{equation}\label{w_1}
 \w_1(s_i-s_j)=s_{\psi(\w_1)(i)}-\s_{\psi(\w_1)(j)}.
 \end{equation}
To prove this, it is enough to treat the case when $\w_1=\si_2$ or $\si_1 \si_2 \si_1$, because $\si_2$ and $\si_1 \si_2 \si_1$ are the generators of $W(\Phi)$. 
Since $\si_2=\si_{s_3-s_1}$ and $\si_1 \si_2 \si_1= \si_{\s_1- \s_2}$, we can check \eqref{w_1} this easily.  In fact, for  
any two roots $\beta $ and $\gamma $,  $\sigma_\beta \gamma$ is given by $\gamma - 2\frac{\beta \cdot \gamma}{\beta \cdot \beta } \beta $ where $\cdot $ is the inner product on $H^2(BT)$ which we defined in section~\ref{sec2}.  
Therefore we have 
\begin{equation*}
\sigma_{\s_3-\s_1}(s_i-s_j) = 
\left\{
\begin{array}{ll}
s_3-s_1 & \text{for } \{ i,j \}=\{ 3,1 \} \\
s_2-s_3 & \text{for } \{ i,j \}=\{ 1,2 \} \\
s_1-s_2 & \text{for } \{ i,j \}=\{ 2,3 \} 
\end{array}
\right.
\end{equation*}
while
\begin{equation*}
s_{\psi(\sigma_{\s_3-\s_1})(i)} - s_{\psi(\sigma_{\s_3-\s_1})(j)} = s_{(3,1)(i)} -s_{(3,1)(j)} =  
\left\{
\begin{array}{ll}
s_3-s_1 & \text{for } \{ i,j \}=\{ 3,1 \} \\
s_2-s_3 & \text{for } \{ i,j \}=\{ 1,2 \} \\
s_1-s_2 & \text{for } \{ i,j \}=\{ 2,3 \} 
\end{array}
\right.
\end{equation*}
up to sign. Thus, \eqref{w_1} holds when $\w_1=\sigma_2=\sigma_{s_3-s_1}$. 
A similar argument proves \eqref{w_1} for $\w_1=\sigma_1 \sigma_2 \sigma_1$. 

Case 2. By assumption  $\w_k= \rho \w_k'$ for $k=1,2$, where $\w_1'$ and $\w_2'$ are in $W(\Phi)$. Remember that $\w_1=\w_2 \si_{\alpha}$ and $\alpha = s_i -s_j$. Since $\rho \w_1'=\rho \w_2' \sigma_\alpha$,   
\begin{equation*}
\psi(\w_1') = \psi(\w_2')\psi(\si_{\alpha}) = \psi(\w_2')(i,j). 
\end{equation*} 
 Therefore $\psi(\w_1')(i)=\psi(\w_2')(j), \psi(\w_1')(j)=\psi(\w_2')(i)$,  
and we have 
$$\ell (e_{w_1 ,w_2}) = \w_1 \alpha=\rho \w_1'(s_i-s_j)=\rho (\s_{\psi(\w_1')(i)} - \s_{\psi(\w_1')(j)}).$$ 
Here $\rho$ preserves $\s_i-\s_j$ up to sign because $\rho$ is the rotation by angle $\pi $, % The proof is straightforward and left to reader again. 
so this completes the proof for case (2).

Case 3. By assumption $w_1$ is in $W(\Phi)$ and $w_2 =\rho w_2 '$ with $w_2 ' \in W(\Phi)$. In this case $\alpha$ is not in $\Phi$, namely $\alpha =\s_{k''}$ for some $k'' \in [3]$. 
 We have  
\begin{equation}\label{lempr3-1}
\w_1=w_2\sigma_\alpha=\rho \w_2' \si_\alpha = \rho \si_{\w_2'(\alpha)} \w_2'.
\end{equation}
This equation means $ \si_{\w _2'(\alpha)} \in \rho W(\Phi)$ because $\w_1$ and $\w_2'$ are in $W(\Phi)$. If $\w _2'(\alpha) \in \Phi$, then  $\si _{\w_2'(\alpha)}$ is  in $W(\Phi)$ and this is contradiction, so there is some $k' \in [3]$ such that 
\begin{equation}\label{lempr3-2}
\w_2'(\alpha) = \s_{k'}
\end{equation}\label{tuika}
up to sign.  
The label of the edge which connects $\w_1$ and $\w_2$ is 
\begin{equation}
\w_1(\alpha) =w_2\sigma_{\alpha}(\alpha)= \rho \w_2'\si_\alpha (\alpha) =\rho w_2'(\alpha) = \rho \s_{k'} = \s_{k'},
\end{equation}  
up to sign.

On the other hand, it follows from \eqref{lempr3-1} and \eqref{lempr3-2} that 
\begin{equation}\label{rho1}
\psi(\w_1) = \psi (\rho \si_{\w_2'(\alpha)} \w_2') = \psi (\rho \si_{\s_{k'}}) \psi ( \w_2'). 
\end{equation} 
Since  
\begin{equation}\label{rho2}
\rho \si_{\s_{k'}}=\si_{\s_{i'}-\s_{j'}} \quad \text{for} \{i', j' \}=[3]\setminus \{ k'\} ,  
\end{equation}
 it follows from \eqref{rho1} and \eqref{rho2} that 
\begin{eqnarray*}
\psi (\w_1) &=& \psi(\rho \si_{s_k'}) \psi(\w_2') = \psi(\si_{s_{i'}-s_{j'}}) \psi(\w_2') \\
           &=& (i',j') \psi (\w_2') = \psi (\w_2') (\psi(\w_2')^{-1}(i'), \psi(\w_2')^{-1}(j')).
\end{eqnarray*}
Let $i:=\psi(\w_2')^{-1}(i')$, $j:=\psi(\w_2')^{-1}(j')$. Then we have $\phi (w_1) (i)=\phi (w_2')(i,j)(i)=\phi (w_2')(j)$ and $\phi (w_1) (j)=\phi (w_2')(i,j)(j)=\phi (w_2')(i)$.
Let $k \in [3] \setminus \{ i,j \}$, namely $k=\psi (w_2')^{-1}(k')$,  
then we have $\psi(w_1)(k)=\phi (w_2')(i,j)(k)= \phi (w_2')(k)=k'$. 
This together with \eqref{lempr3-2} completes the  proof of case (3).  
\end{proof}
Using the above lemma, we can redescribe the labeled graph associated with the root system $\Phi(G_2)$, denoted by $\mG_2$, as follows; \\
\begin{itemize}
\item The vertex set $V(\mG_2)$ is $\{(v, \varepsilon) \  \mid \  v \in S_3, \varepsilon = + ,-\}$. 
\item $\w_1=(v_1,\varepsilon_1) $ and $\w_2=(v_2,\varepsilon_2) $ are connected by an edge   $e_{\w_1,\w_2}$ if and only if there are some integers $i$ and $j$ such that $v_1(i)=v_2(j), v_1(j)=v_2(i)$. 
\item The label of the edge $e_{\w_1,\w_2}$ is $s_{v_1(i)}-s_{v_1(j)}$ if $\varepsilon_1 = \varepsilon_2$, and $s_{v_1(k)}$ if $\varepsilon_1 \not = \varepsilon_2$ where $k \in [3] \setminus \{ i,j \}$.
\end{itemize}
See Figure~\ref{g2}.
\begin{figure}[h]
\begin{minipage}{0.6\hsize}
\includegraphics[width=70mm]{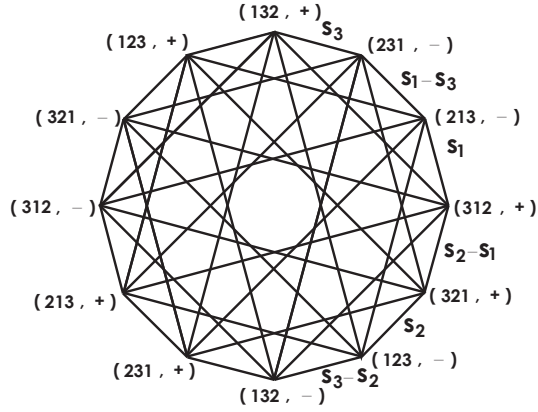}\caption{The labeled graph $\mathcal{G}_2$} \label{g2}
\end{minipage}
\end{figure}
Note that in Figure~\ref{g2}, the parallel edges have the same label. 
\begin{rema}
Let $V^-$ (resp. $V^+$) be a subset of $V(\mG_2)$ defined to be $\{(v, \varepsilon) \mid v \in S_3, \varepsilon = -  \text{(resp. $+$)} \}$, and 
 $\mG_{-}$ (resp. $\mG_{+}$) be the labeled full subgraph of $\mG_2$ with $V^{-}$ (resp. $V^+$) as a vertex set. Clearly $\mG_{-}$ and $\mG_{+}$ are isomorphic as a labeled graph to the labeled graph $\mA_3$ associated with the root system of type $A_2$. 
Guillemin, Sabatini and Zara introduced the notion of a GKM fiber bundle in \cite{gu-sa-za} (in a GKM fiber bundle, the total space, fiber and base space are  all labeled graphs). 
In fact, it can be seen that $\mG_2$ is the total space of a GKM fiber bundle, with fibers isomorphic to $\mA_2$. In this sense it is natural to expect that the result of type $A_2$ plays a role when we determine the ring structure of $\mHT^*(\mG_2)$ below. 
\end{rema}

\section{Graph cohomology ring of $\mG_2$}\label{cohomology}
In this section, we state our main result which describes the ring structure of the graph cohomology ring of the labeled graph $\mG_2$. 
We review the definition of the graph cohomology ring of a labeled graph first. 

\begin{defi} 
\begin{sloppypar}
Let $\mathcal{G}$ be a labeled graph with a label $\ell$ taking values in $H^2(BT)$ and let $W$ be the vertex set of $\mathcal{G}$. 
We identify $\bigoplus _{w \in W} H^*(BT)$ with $\Map(W, H^*(BT))$ where $Map(W, H^*(BT))$ is the set of all maps from $W$ to $H^*(BT)$.
Then the graph cohomology $H^*(\mathcal{G})$ of $\mathcal{G}$ is defined to be the set of all $h \in \Map(W, H^*(BT))$ which  satisfies the so-called \lq\lq GKM condition", namely for any two vertices $\w$ and $\w'$ connected by an edge $e$, $h(\w)-h(\w') $ is divisible by $\l(e)$. The ring structure on $H^*(BT)$ induces a ring structure on $H^*_T(\mathcal{G})$.   
\end{sloppypar} 
\end{defi} 
To become familiar with a graph cohomology ring, we shall remember the graph cohomology ring of $\mA_3$. 
\begin{exam}
We define  elements of $\mHT^*(\mA_3)$ denoted by $\tau_i$'s and $\t_i$'s as follows;  
  \[ \tau_i(v):= \t_{v(i)} \  \  \  \  \  \t_i(v):=\t_i \  \  \  \  \  \text{for any } v \in S_3, \ \  i=1,2,3.  \]
We regard an element in $\Map(W(\mA_3), \Z[\t_1, \t_2 ,\t_3])$ as a set of six polynomials in $\t_1, \t_2, \t_3$ such that each polynomial corresponds to some vertex, because $W(\mA_3)=S_3$ has six verticies. 
So the elements $\tau_i$'s are described as the following figures.
\begin{center}
\ 
\begin{picture}(90,50)
 \linethickness{1pt}
 \put(5,0){\line(100,173){15}}
 \put(20,25.95){\line(1,0){30}}
 \put(50,25.95){\line(100,-173){15}}
 \put(65,0){\line(-100,-173){15}}
 \put(50,-25.95){\line(-1,0){30}}
 \put(5,0){\line(100,-173){15}}
 \put(5,0){\line(1,0){60}}
 \put(20,25.95){\line(200,-346){30}}
 \put(20,-25.95){\line(200,346){30}}
  \put(-10,0){$\t_2$}
  \put(2,28){$\t_1$}
  \put(50,28){$\t_1$}
  \put(67,0){$\t_3$}
  \put(50,-38){$\t_3$}
  \put(15,-38){$\t_2$}
   \put(30,-55){$\tau_1$}
\end{picture}
\ 
\begin{picture}(90,50)
\linethickness{1pt}
 \put(5,0){\line(100,173){15}}
 \put(20,25.95){\line(1,0){30}}
 \put(50,25.95){\line(100,-173){15}}
 \put(65,0){\line(-100,-173){15}}
 \put(50,-25.95){\line(-1,0){30}}
 \put(5,0){\line(100,-173){15}}
 \put(5,0){\line(1,0){60}}
 \put(20,25.95){\line(200,-346){30}}
 \put(20,-25.95){\line(200,346){30}}
  \put(-10,0){$\t_1$}
  \put(2,28){$\t_2$}
  \put(50,28){$\t_3$}
  \put(67,0){$\t_1$}
  \put(50,-38){$\t_2$}
  \put(15,-38){$\t_3$}
   \put(30,-55){$\tau_2$}
\end{picture} 
\ 
\begin{picture}(90,50)
 \linethickness{1pt}
 \put(5,0){\line(100,173){15}}
 \put(20,25.95){\line(1,0){30}}
 \put(50,25.95){\line(100,-173){15}}
 \put(65,0){\line(-100,-173){15}}
 \put(50,-25.95){\line(-1,0){30}}
 \put(5,0){\line(100,-173){15}}
 \put(5,0){\line(1,0){60}}
 \put(20,25.95){\line(200,-346){30}}
 \put(20,-25.95){\line(200,346){30}}
  \put(-10,0){$\t_3$}
  \put(2,28){$\t_3$}
  \put(50,28){$\t_2$}
  \put(67,0){$\t_2$}
  \put(50,-38){$\t_1$}
  \put(15,-38){$\t_1$}
   \put(30,-55){$\tau_3$}
\end{picture}
\ \\ \  \\ \  \\ \  \\ \  \\  \  
\end{center}   
One can easily check that $\tau_i$'s are in $H^*_T(\mA_3)$. 
One can also check that the elements $\t_i$'s and $\tau_i$'s generate $H^*_T(\mA_3)$ as a ring, and  
\[ \mHT^*(\mA_3) \cong \Z[\tau_1, \tau_2, \tau_3, \t_1, \t_2, \t_3]/J \]
where $J=<e_i(\tau)-e_i(\t)  \mid i=1,2,3>$ and $e_i(\tau)$ (resp. $e_i(\t)$)  is the $i^{th}$-elementary symmetric polynomial in $\tau_1, \tau_2, \tau_3$ (resp. $\t_1, \t_2, \t_3$). (See \cite{fu-i-ma}.)
\end{exam}
\begin{sloppypar}
Now we consider elements in $H^*_T(\mG_2)$. 
We set $\s_{w(i)}:=\varepsilon\s_{v(i)}$ for $\w=(v, \varepsilon) \in V(\mG_2)$.
For each $i=1,2,3$, we define elements $\tau_i$, $\t_i$ of $\Map(V(\mG_2), \Z[\t_1, \t_2, \t_3])$ by 
\end{sloppypar}
\begin{equation}\label{tau}
\tau_i(\w):=\s_{\w(i)}=\varepsilon\s_{v(i)}, \ \ \ \t_i(\w):=\t_i \ \ \ \text{for} \ \ \  \w=(v,\varepsilon) \in V(\mG_2 ).
\end{equation}
One can check that $\tau_i$'s, $\t_i$'s are elements of $H^*_T(\mG_2)$.   
We set $\s_i:= \t_i -\t_{i+1}$ in $\Map(V(\mG_2),\Z[\t_1, \t_2, \t_3])$, namely $\s_i(\w)=(t_i-t_{i+1})(\w)=t_i-t_{i+1}=\s_i$, where $\t_4=\t_1$. 
In addition, we define an element $f$ of $\Map(V(\mG_2),\Z[\t_1, \t_2, \t_3])$ by
\begin{equation}\label{fdef}
f(\w) :=\left\{
\begin{array}{ll}
  s_1 s_2 s_3& \text{for} \ \ \  \w=(v, +)\\
  0 & \text{for} \ \ \  \w=(v, -) . \\
\end{array}
\right. 
\end{equation}
Clearly, $f$ is also an element of $H^*_T(\mG_2)$.
Elements $\tau_1$ and $f$ are described in Figures~\ref{figtau} and \ref{figf}.
\begin{figure}[h]
\begin{minipage}{0.4\hsize}
\includegraphics[width=48.1mm]{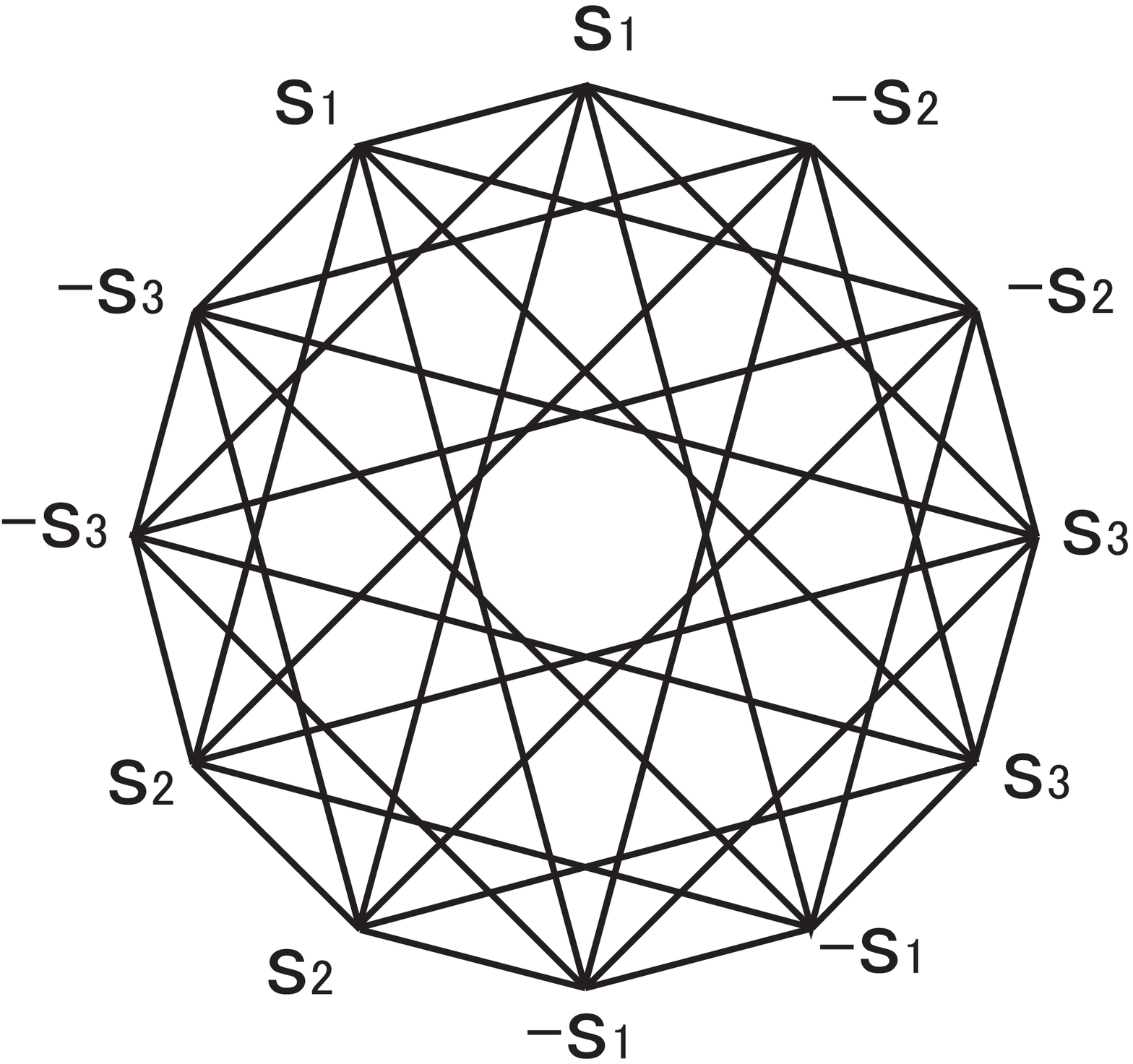}\caption{$\tau_1$}\label{figtau}
\end{minipage}\ \ \ \ \ \ 
\begin{minipage}{0.4\hsize}
\includegraphics[width=55.5mm]{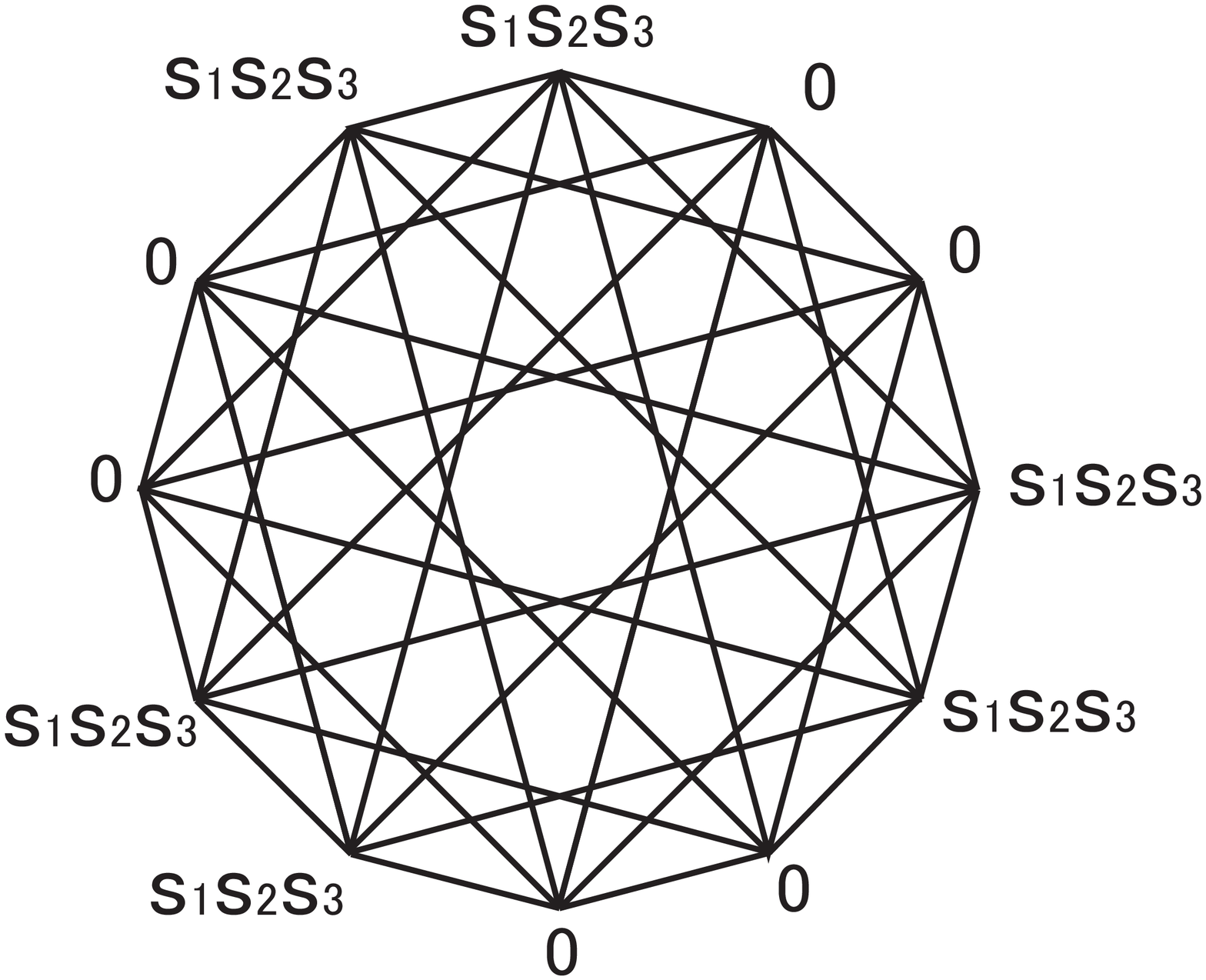}\caption{$f$}\label{figf}
\end{minipage}
\end{figure}

\begin{rema}
The restriction of $\tau_i$  to the subgraph $\mG_{+}$ (resp. $\mG_{-}$)  is $\tau_i$ (resp. $-\tau_i$)  in  $\mHT^*(\mA_3)$. 
In other words $\tau_i$ in $H^*_T(\mG_2)$ is a lift of the $\tau_i$ in $H^*_T(\mA_3)$.  
The element $f$ comes from some element in the graph cohomology ring of the labeled graph of the base space.  In fact, the labeled graph of the base space, denoted by $\mathcal{B}$, is described as  Figure~\ref{base}, and Figure~\ref{baseno} describes an element in  $H^*_T(\mathcal{B}) \subset \Map(V(\mathcal{B}), \Z[\t_1,\t_2,\t_3])$. The element $f$ in $H^*_T(\mG_2)$ is the pullback of the element in Figure~\ref{baseno} by the projection $\mG_2\to \mathcal{B}$. 
\begin{figure}[h]
\begin{minipage}{0.4\hsize}
\includegraphics[width=50mm]{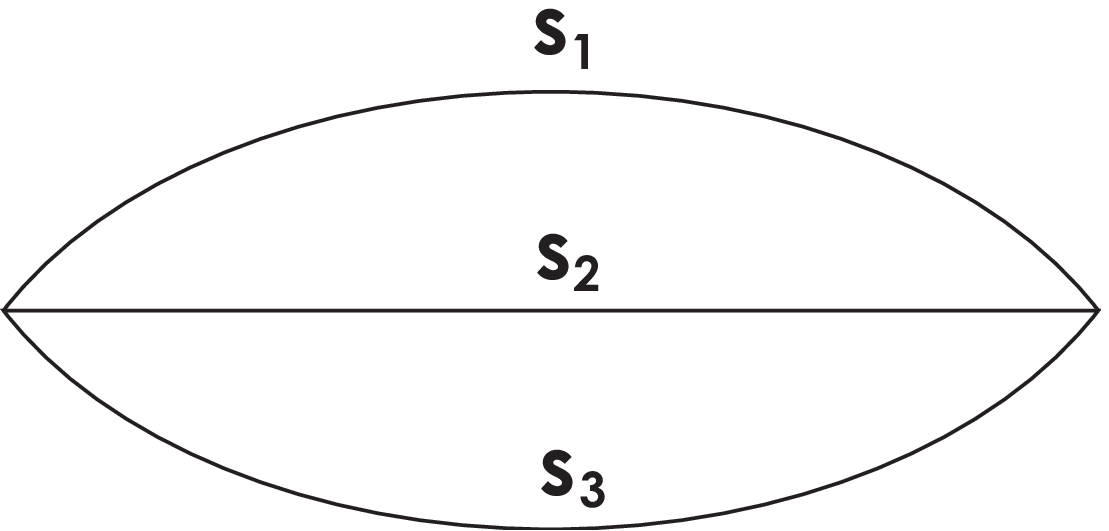}\caption{}\label{base}
\end{minipage} \quad
\begin{minipage}{0.4\hsize}
\includegraphics[width=50mm]{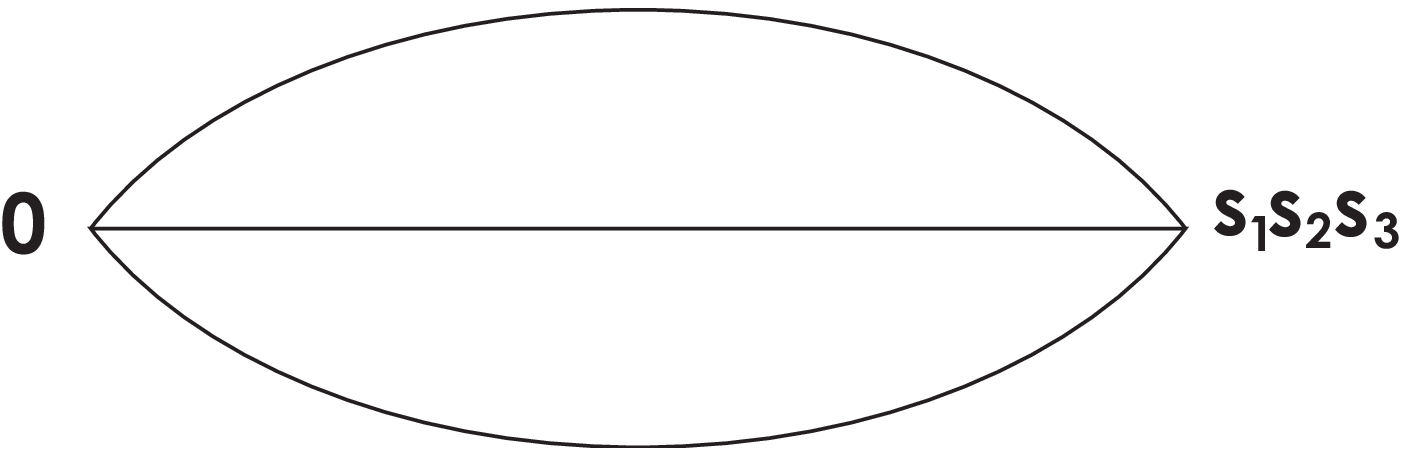}\caption{}\label{baseno}
\end{minipage}
\end{figure}

Guillemin, Sabatini and Zara \cite{gu-sa-za} construct module generators of the graph cohomology of the total space over the graph cohomology of the base space, by using module generators of the graph cohomology of fiber. (They consider the graph cohomology with $\R$ coefficient.)   
\end{rema}
\begin{theo}\label{maintheo}
Let $\mG_2$ be the labeled graph associated with the root system $\Phi(G_2)$ of type $G_2$. Then 
\begin{equation*}
H^*_T(\mG_2) = \Z[\tau_1,\tau_2,\tau_3,\t_1,\t_2,\t_3,f]/I,
\end{equation*}
where $I=<e_1(\tau), e_2(\tau)-e_2(\s), 2f-e_3(\tau)-e_3(\s), f^2-fe_3(\s) >$, and $e_i(\tau)$ $($resp. $e_i(s)$ $)$ is the $i^{th}$ elementary symmetric polynomial in $\tau_1,\tau_2,\tau_3$ 
$($resp. $s_1,s_2,s_3$ $)$.
\end{theo}

The rest of the paper is devoted to the proof of Theorem~\ref{maintheo}.
\begin{lemm}\label{surj}
$H^*_T(\mG_2)$ is generated by $\tau_1, \tau_2, \tau_3$, $\t_1, \t_2, \t_3$ and $f$ as a ring.
\end{lemm}
\begin{proof}
The idea of the proof of the lemma is same as that of Lemma 3.2 in \cite{fu-i-ma}. 
\begin{claim}
For any homogeneous element $h$ of $H^*_T(\mG_2)$, there is a polynomial $G$ in $\tau_i$'s and $\t_i$'s 
such that the restrictions of $h$ and $G$ to the subgraph $\mG_{-}$ coincide. 
\end{claim}

\medskip
\noindent
Proof of Claim. 
We set $$V_{i}^{-} :=\{ (v,\varepsilon ) \in S_3 \times \{ \pm \} \mid v(i)=3 , \varepsilon =- \} \quad \text{for} \quad i\in [3].$$ % and $\mathcal{L}_{i}^{-}$ be the labeled full subgraph with $V_{i}^{-}$ as the vertex set.
Let $2k$ be the degree of $h \in H^*_T(\mG_2)$ and 
\begin{equation*}
1 \leq q \leq \min \{ k+1, 3 \}  
\end{equation*} 
and assume that 
\begin{equation*}
h(\w_i )=0 \ \ \ \text{for any } \w_i \in V_{i}^{-} \ \text{whenever} \ \ i < q.
\end{equation*}
For any $\w =(v, -) \in V_{q}^{-}$, there is a unique vertex $\w_i =(v_i ,-)\in V_{i}^{-}$ for each $1 \leq i < q$ such that $\w$ and $\w_i$ are connected by an edge labeled by $\s_{\w(i)}-\s_{\w(q)}=\s_{\w(i)}+\s_3$. (Namely $v=v_i(i,q)$.)
Then $h(\w)-h(\w_i)=h(\w)$ is divisible by $\s_{\w(i)}+\s_3$ for $i < q$, so there is a homogeneous polynomial $g_{\w} \in \Z[\t_1,\t_2,\t_3]$ of degree $2(k-q+1)$ such that 
\begin{equation}\label{inj1}
h(\w )=g_{\w}\prod_{i=1}^{q-1}(\s_{\w (i)}+\s_3 ).
\end{equation} 
On the other hand, since $(\tau_k+\s_3)(\u)=\s_{\u(k)}+\s_3$ for any $\u \in V(\mG_2)$, 
\begin{equation}\label{lem1.1}
\prod_{i=1}^{q-1}(\tau_i+\s_3)(\u)=
%  0 & \text{for} \ \ \  \u \in V_{i}^{-} \ \ \text{whenever} \ \ i <q \\
  \prod_{i=1}^{q-1}(\s_{\u (i)}+\s_3 ).   
\end{equation}
In particular, when $\u$ is in $ V_{i}^{-}$ for $i <q$, \eqref{lem1.1} is equal to $0$ since $\s_{u(i)}=-s_3$.  
So, it follows from \eqref{inj1} and \eqref{lem1.1} that 
\begin{equation}\label{inj2}
\left( h-g_{\w} \prod_{i=1}^{q-1}(\tau_i+\s_3)  \right) (\u)=0 ,
\end{equation}
whenever $\u \in V_{i}^{-} $ for $i<q $ or $\u=\w$.
We set 
\begin{equation}\label{inj3}
h':= h-g_{\w} \prod_{i=1}^{q-1}(\tau_i+\s_3) .
\end{equation} 
Note that 
\begin{equation}\label{lem1.3}
h'(u)=0\quad \text{if $u\in V_i^-$ for $i<q$ or $u=w$.}
\end{equation}

Let $\w'$ be the other vertex in $V_q^{-}$.  (Namely, $V_q^- = \{ \w, \w' \}$.) 
Then $\w$ and $\w'$ are connected by an edge labeled by $\s_1-\s_2$ and there is a unique vertex $\w_i'=(v_i',-) \in V_i^-$ for each $1 \leq i < q$ such that $\w'$ and $\w_i'$ are connected by an edge labeled by $s_{\w'(i)}-s_{\w'(q)}=s_{\w'(i)}+s_3$, so $h'(\w')-h'(\w)=h'(\w')$ is divisible by $\s_1-\s_2$ and $s_{\w'(i)}+s_3$.   
Therefore, there is a homogeneous polynomial $g_{\w'} \in \Z[\t_1,\t_2,\t_3]$ of degree $2(k-q)$ such that 
\begin{equation}\label{h'1}
h'(\w')=(\s_1-\s_2) g_{\w'} \prod_{i=1}^{q-1}(\s_{\w' (i)}+\s_3 ) .
\end{equation} 
On the other hand, for $j \in [3] \setminus \{ q \}$,
\begin{eqnarray}
\left( \tau_{j}-\s_{\w(j)} \right)(\w')&=&\s_{\w'(j)}-\s_{\w(j)}=\delta (\s_1-\s_2) \label{h'2}\\
\left( \tau_{j}-\s_{\w(j)} \right)(\w)&=&\s_{\w(j)}-\s_{\w(j)}=0 \label{h'3},
\end{eqnarray}
where $\delta = \pm 1$ and $\delta$ depends on $\w$. 
Thus, it follows from \eqref{lem1.1}, \eqref{h'1}, \eqref{h'2} and \eqref{h'3} that 
\begin{equation*}
\left( h'-(\tau_j -\s_{\w(j)})(\delta g_{\w'}) \prod_{k=1}^{q-1}(\tau_k+\s_3)  \right) (\u)=0.  
\end{equation*}
for $\u \in V_{i}^{-}$ whenever $i \leq q$.
Therefore, putting $H=g_{\w}+(\tau_j -\s_{\w(j)})(\delta g_{\w'}) $, and subtracting the polynomial 
$  H \Pi _{k=1}^{q-1}(\tau_j+\s_3) $ 
from $h$, we may assume that
$$h(\u)=0 \ \ \ \text{for any } \ \ \u \in V_{i}^{-} \ \ \ \text{whenever} \ \ \ i < q+1.$$

The above argument implies that  $h$ finally takes zero on all vertices in $V^-$ by subtracting a polynomial in $\tau_i$'s and $\t_i$'s, and this completes the proof of the claim.
\\ \ 

The claim allows us to assume that our homogeneous element $h$ in $H^*_T(\mG_2)$ satisfies $h(\u)=0$ for all $\u \in V^-$.
Any $\w =(v, +) \in V^+$ has a unique edge which connects $\w$ and some $\w_i=(v_i, -) \in V_{i}^{-}$ for each $i \in [3]$, and the edge has a label $\s_{\w(k)}$ where $k$ is determined by $v(k)=v_i(k)$. 
Namely, $h(\w)-h(\w_i)=h(\w)$ is divisible by $\s_{\w(k)}$, thus, there is a homogeneous polynomial $p_{\w}$ of  $\Z[\t_1, \t_2, \t_3]$ such that
$$
h(\w)=\s_1 \s_2 \s_3 p_{\w}.
$$ 
In addition the collection of polynomials $\{p_{\w}\}$ satisfies the GKM condition in $\mG_+$. 
In fact, for any vertices $\w$ and $\w'$ in $V^+$ connected by an edge with label $s_i-s_j$ for some $i,j \in [3]$, it follows from the definition of  $H^*_T(\mG_2)$ that $h(\w)-h(\w')=\s_1 \s_2 \s_3(p_{\w}-p_{\w'})$ is divisible by $s_i-\s_j$. So $p_{\w}-p_{\w'}$ is divisible by the label $\s_i-\s_j$. 
 Then the same argument as in the proof of the claim above shows that there is a polynomial $H'$ in $\tau_i$'s and $\t_i$'s such that $H'(\w)=p_{\w}$ for $\w \in V^+$.  
Therefore, 
\begin{equation*}
(h-H'f)(\w)=0 \ \ \ \text{for any} \ \ \ \w \in V(\mG_2).
\end{equation*}
This completes the proof of the lemma.
\end{proof}

Remember that the Hilbert series of a graded ring  $A^*=\oplus_{j=0}^\infty A^j$, where $A^j$ is the degree $j$ part of $A^*$ and of finite rank  over $\ \mathbb{Z}$, is a formal power series defined by 
\[
F(A^*,x):= \sum_{j=0}^\infty(\rank_{\mathbb{Z}} A^j)x^j.\]

\begin{lemm}\label{inje}
$$F(H^*_T(\mG_2),x)= \frac{1}{(1-x^2)^3}(1+x^2)(1+x^2+x^4)(1+x^6)$$
\end{lemm}
\begin{proof}
We set $d(k):=\rank H^{2k}_T(\mG_2)$ and $r(k):=\rank \Z[\t_1, \t_2, \t_3]^{2k}$, where $\Z[\t_1, \t_2, \t_3]^{2k}$ is the set of homogeneous polynomials of degree $2k$. We note that the degree of $\t_i$ is $2$. 
For $h \in H^{2k}_T(\mG_2)$, assume that there is some $q \in [3]$ such that $$h(\w)=0 \ \ \text{for} \ \  \w \in V^-_i \ \ \text{whenever} \ \  i < q.$$
Then, the proof of the claim in Lemma~\ref{surj} shows that there is a polynomial $H$ in $\tau_i$'s and $\t_i$'s such that $$(h-H)(\w)=0 \quad \text{for} \quad \w \in V^-_i \quad \text{whenever} \ \ i \leq q,$$
and that the polynomial $H$ is of the form % by using the way in the proof of the claim in the Lemma~\ref{surj}. 
 
$$H= \left(g_\w+\delta(\tau_j-s_{w(j)})g_{\w'} \right)\prod^{q-1}_{k=1}(\tau_k+s_3)$$ 
where $g_\w$ and $g_{\w'}$ are some polynomials in $\t_1, \t_2, \t_3$, and the degree of $g_\w$ (resp. $g_{\w'}$) is $2(k-q+1)$ (resp. $2(k-q)$).  
Therefore the rank of the additive group consisting of all such polynomials $H$ is given by $$r(k-q+1)+ r(k-q).$$

Similarly, assume that there is some $q \in [3]$ such that $$h(\w)=0 \ \ \text{for} \ \  \w \in V^- \ \ \text{or} \ \ \w \in V^+_i \ \ \text{whenever} \ \  i < q,$$
where $V^+_i=\{ (v, +) \  \mid  \ v(i)=3 \}.$
Then, the proof of Lemma~\ref{surj} shows that there is a polynomial $H'$ in $\tau_i$'s, $\t_i$'s and $f$ such that $$(h-H')(\w)=0 \quad \text{for} \quad \w \in V^- \quad \text{or} \quad \w \in V^+_i \quad \text{whenever} \quad i \leq q,$$
and that the polynomial $H'$ is of the form  
$$H'=s_1 s_2 s_3  \left(g'_\w+\delta(\tau_j-s_{w(j)})g'_{\w'}\right)\prod^{q-1}_{k=1}(\tau_k+s_3)$$ where $w$ and $w'$ are the two vertices of $V^+_q$ and $g'_\w$ and $g'_{\w'}$ are some polynomials in $\t_1, \t_2, \t_3$, and the degree of $g'_\w$ (resp. $g'_{\w'}$) is $2(k-2-q)$ (resp. $2(k-3-q)$).
Thus, the rank of the the additive group consisting of all such polynomials $H'$ is given by $$r(k-2-q)+ r(k-3-q).$$

Let $G^{+}_q$ (resp. $G^{-}_q$) be the the additive group consisting of all polinomials $h$ of degree $2k$ such that 
$$h(\w)=0 \ \ \text{for} \ \  \w \in V^- \ \ \text{or} \ \ \w \in V^+_i \ \ \text{whenever} \ \  i < q$$
$$\text{(resp. } h(\w)=0 \ \ \text{for} \ \  \w \in V^-_i\text{whenever} \ \  i < q \text{\  )}.$$
Then, the rank of $G^{+}_3$ is $r(k-5)+r(k-6)$ by the above argument. 
The rank of $G^{+}_2$ is equal to $\rank G^{+}_3$ plus the rank of additive group consisting of all polynomial $H'$ such that 
$$(h-H')(\w)=0 \quad \text{for} \quad \w \in V^- \quad \text{or} \quad \w \in V^+_i \quad \text{whenever} \quad i \leq 2,$$
so the rank of $G^{+}_2$ is $\rank G^+_3+r(k-4)+r(k-5)$.
Similarly, the rank of $G^{+}_1$ is equal to  $\rank G^{+}_2$ plus $r(k-3)+r(k-4)$, namely 
\begin{equation}\label{lempr4.7-1}
\rank G^+_1=\sum_{i=4}^{6}(r(i)+r(i-1)).
\end{equation} 
In the same way, we have 
\begin{equation}\label{lempr4.7-2}
\rank G^-_3= \rank G^+_1 + r(k-2)+r(k-3),
\end{equation}
and 
\begin{equation}\label{lempr4.7-3}
\rank G^-_{q} = \rank G^-_{q+1} + r(k-q+1)+ r(k-q)  \quad \text{for } q=1,2 .
\end{equation}
Therefore, it follows from \eqref{lempr4.7-1}, \eqref{lempr4.7-2} and \eqref{lempr4.7-1} that 
$$d(k)=\rank G^-_1 = \sum_{i=k}^{k-5} \Big( r(i)+r(i-1) \Big)= r(k)+2\sum_{i=k-5}^{k-1}r(i) + r(k-6)$$ where $r(i)=0$ for $i<0$.
Namely, 
\begin{equation*}
d(k)=
\left\{
\begin{array}{ll}
  r(0) & \text{for} \ \ \ k=0 \\
  2\sum _{i=0}^{k-1}r(i) + r(k) & \text{for} \ \ \  1 \leq k \leq 5\\
  r(k-6) + 2\sum _{i=k-5}^{k-1}r(i) + r(k) & \text{for} \ \ \  k > 5 . \\
\end{array}
\right. 
\end{equation*}
Therefore, 
\begin{eqnarray*}
F(H^*_T(\mG_2 ),x) &=& \sum_{k=0}^{\infty} d(k) x^{2k}\\
                   &=& r(0) + \sum_{k=1}^{5}\left( 2\sum _{i=0}^{k-1}r(i) + r(k) \right) x^{2k} \\
                   & \  & \quad + \sum_{k=6}^{\infty} \left( r(k-6) + 2\sum _{i=k-5}^{k-1}r(i) + r(k) \right) x^{2k}\\
                   &=& \sum_{k=0}^{\infty}r(k)x^{2k} + \sum_{k=6}^{\infty}r(k-6)x^{2k} \\ 
                   & \  & \quad+ 2\left( \sum_{k=1}^{5} \sum_{i=0}^{k-1}r(i)x^{2k}+\sum_{k=6}^{\infty} \sum_{i=k-5}^{k-1}r(i)x^{2k} \right) \\
                   &=& (1+ x^{12})(r(0)+r(1)x^2+r(2)x^4+ \cdots) \\
                   & \  & \  + 2(x^{2}+x^4+x^6+x^8+x^{10})(r(0)+r(1)x^2+r(2)x^4+ \cdots )\\
                   &=& (1+ 2x^{2}+2x^4+2x^6+2x^8+2x^{10}+x^{12})F(\Z[\t_1, \t_2, \t_3], x)\\
                   %&=& (1+x^6)(1+2x^2+2x^4+x^6)F(\Z[\t_1, \t_2, \t_3], x)\\
                   &=& (1+x^2)(1+x^2+x^4)(1+x^6)F(\Z[\t_1, \t_2, \t_3], x)\\
                   &=& \frac{1}{(1-x^2)^3}(1+x^2)(1+x^2+x^4)(1+x^6),
\end{eqnarray*}
 proving the lemma.
\end{proof}

We abbreviate the polynomial ring $\mathbb{Z}[\tau_1,\tau_2,\tau_3,\t_1,\t_2,\t_3,f]$ as $\Z[\tau,\t,f]$. The canonical map $\Z[\tau,\t,f] \to H^*_T(\mG_2)$ is a grade preserving homomorphism which is surjective by Lemma~\ref{surj}.  %Let $e_i(\tau)$ (resp. $e_i(\s)$) denote the $i^{th}$ elementary symmetric polynomial in $\tau_1,{\scriptstyle\cdots},\tau_n$ (resp. $\t_1,{\scriptstyle\cdots},\t_n$).  
It easily follows from \eqref{tau} and \eqref{fdef} that 
\begin{eqnarray}
e_1(\tau) &=&0, \label{1} \\ 
e_2(\tau)-e_2(\s) &=&0,  \label{2}\\
2f-e_3(\tau)-e_3(\s) &=&0, \label{3} \quad \quad \text{and} \\
f^2-fe_3(\s) &=&0. \label{4}
\end{eqnarray} 
Therefore the canonical map above induces a grade preserving epimorphism 
\begin{equation} \label{cano}
\Z[\tau,\t,f]/I \to H_T^*(\mG_2), 
\end{equation}
where $$I=<e_1(\tau), e_2(\tau)-e_2(\s), 2f-e_3(\tau)-e_3(\s), f^2-fe_3(\s)>.$$
We note that $\Z[\tau,\t,f]/I$ is a $\Z[t]$-module in a natural way. 
\begin{lemm}\label{gene}
$\Z[\tau,\t ,f]/I$ is generated by $\prod_{k=1}^3\tau_k^{i_k}f^j$ for $0 \leq i_k \leq 3-k$ and $j=0,1$ as a $\Z[\t]-$module.
\end{lemm}

\begin{proof}
Clearly the elements $\Pi_{k=1}^{3}\tau_{k}^{i_k}f^{j}$ with no restriction on exponents, generate $\Z[\tau,\t ,f]/I$ as a $\Z[\t]-$module.
The identity \eqref{1}  means  
\begin{equation}\label{tau3}
 \tau_3=-\tau_1-\tau_2. 
\end{equation}
By using \eqref{tau3}, we have 
$$e_2(\tau)=\tau_1 \tau_2+\tau_3 (\tau_1 +\tau_2 )=\tau_1 \tau_2 -(\tau_1+\tau_2)^2=-\tau_1^2 -\tau_2^2-\tau_1 \tau_2.$$
Therefore, $
\tau_2^2=-\tau_1^2-\tau_1\tau_2-e_2(\tau),
$
and hence 
\begin{equation}\label{tau2}
\tau_2^2=-\tau_1^2-\tau_1 \tau_2-e_2(s)
\end{equation}
by \eqref{2}. 
It follows from \eqref{tau3} and \eqref{tau2} that
\begin{eqnarray*}
e_3(\tau) &=& \tau_1\tau_2 \tau_3 \\
          &=& -\tau_1\tau_2(\tau_1+\tau_2) \\
          &=& -\tau_1^2\tau_2-\tau_1\tau_2^2 \\
          &=& -\tau_1^2\tau_2+\tau_1(\tau_1^2+\tau_1 \tau_2 +e_2(s))\\
          &=& \tau_1^3+\tau_1 e_2(s),\\
\end{eqnarray*}
Therefore $\tau_1^3=-\tau_1 e_2(s)+e_3(\tau)$ and hence 
\begin{equation}\label{tau1}
\tau_1^3=-\tau_1 e_2(s)+2f-e_3(s)
\end{equation}
by \eqref{3}. 
Therefore, $\tau_k^{3-k+1}$ is written as \eqref{tau3}, \eqref{tau2} and \eqref{tau1}%the polynomial in $\tau_l$ ($l \leq k$) and $\t_i$'s
, satisfying that the exponent of $\tau_k$ is less than or equal to $3-k$ and that of $f$ is 0 or 1. 
In addition, $f^2$ is written as \eqref{4}, so we can always assume the exponent of $f$ to be 0 or 1. This completes the proof of the lemma.
\end{proof}

Now we are in a position to complete the proof of Theorem~\ref{maintheo}. 
\begin{proof}[Proof of Theorem~\ref{maintheo}]  
If two formal power series $a(x)=\sum_{i=0}^\infty a_ix^i$ and $b(x)=\sum_{i=0}^\infty b_ix^i$ in $x$ with real coefficients $a_i$ and $b_i$ satisfy $a_i\le b_i$ for every $i$, then we express this as $a(x)\le b(x)$.  

The Hilbert series of the free $\Z[t]$-module generated by $\prod_{k=1}^{3}\tau_k^{i_k}f^j$ is given by $\frac{1}{(1-x^2)^3}x^{2(\sum_{k=1}^{3}i_k+3j)}$, so it follows from Lemma~\ref{gene} that 
\begin{equation} \label{ineq}
F(\Z[\tau,\t,f]/I,x)\le \frac{1}{(1-x^2)^3}\sum_{0 \leq i_1 \leq 2, \ 0 \leq i_2 \leq 1}x^{2(i_1 +i_2)}\sum_{j=0}^{1}x^{6j}
\end{equation}
and the equality holds above if and only if $\Z[\tau,\t,f]/I$ is free as a $\Z[t]$-module.  Here the right hand side in \eqref{ineq} above is equal to 
\[
\frac{1}{(1-x^2)^3}(1+x^2)(1+x^2+x^4)(1+x^6)%=\frac{1}{(1-x^2)^3}\prod_{i=1}^n(1-x^{2i})
\]
which agrees with $F(H^*_T(\mG_2),x)$ by Lemma~\ref{inje}.  Therefore $F(\Z[\tau,\t,f]/I,x) \le F(H^*_T(\mG_2),x)$.  On the other hand, the surjectivity of the map \eqref{cano} implies the opposite inequality.  Therefore  $F(\Z[\tau,\t,f]/I,x)=F(H^*_T(\mG_2),x)$.  This means that the inequality in \eqref{ineq} must be an equality and hence $\Z[\tau,\t,f]/I$ is free as a $\Z[t]$-module, in particular, as a $\Z$-module.  Since the map in \eqref{cano} is surjective and $F(\Z[\tau,\t,f]/I,x)=F(H^*_T(\mG_2),x)$, we conclude that the map in \eqref{cano} is actually an isomorphism. This proves Theorem~\ref{maintheo}.
\end{proof}

\bigskip
\noindent
{\bf Acknowledgment.}   The author takes this opportunity to thank her Ph.D. supervisor Mikiya Masuda for sparing his precious time and give constructive comments throughout the process of writing this paper. I also thank Shizuo Kaji and Hiroshi Naruse, who gave me the opportunity to consider this problem on the flag manifold of type $G_2$. Finally, Megumi Harada gave me lots of valuable advice both about mathematics and writing in English.

\bigskip
%\textcolor{red}{「考えたいこと」$\tau_i$, $f$の幾何学的意味または実現．それが分かれば，$H^*_T(\mG_2)$が$G_2$flagの同変コホモロジーと同型であることが示せる．出来ればここまでやりたい．「注意」$G_2$はランク２なので，$G_2/T$と書いた場合の$T$はランク２だが，我々が使っている$T$はランク３．}

\end{document}